\documentclass[12pt]{amsart}

\usepackage{amsfonts, amssymb}

\newtheorem{theorem}{Theorem}[section]

\newtheorem{proposition}[theorem]{Proposition}

\theoremstyle{definition}
\newtheorem{definition}[theorem]{Definition}
\theoremstyle{remark}

\numberwithin{equation}{section}

\newcommand{\CC}{\mathbb C}
\newcommand{\RR}{\mathbb R}

\newcommand{\dist}{\text{\rm dist }}

\newcommand{\fddtheta}[1]{\frac{\partial #1}{\partial \theta}}

\def\NN{{\mathbb N}}

\def\PP{{\mathbb P}}

\def\cC{{\mathcal C}}

\def\cF{{\mathcal F}}

\def\cL{{\mathcal L}}

\def\dist{\textrm{dist}\,}

\begin{document}

\title[Brody curves in complicated sets]{Brody curves in complicated sets}%

\author{Taeyong Ahn}%
\address{(Ahn) Center for Geometry and its Applications, 
POSTECH, Pohang 790-784 The Republic of Korea}%
%
\email{(Ahn) triumph@postech.ac.kr}
\date{\today}
\thanks{Research of the author is 
supported in part by the grant 2011-0030044 (The SRC-GAIA)
of the NRF of Korea.}

\begin{abstract}
For a hyperbolic generalized H\'enon mapping (in the sense of \cite{BS1991}), $J^+$, the boundary of the set of non-escaping points, is known as a complicated set and also known to admit a foliation by biholomorphic images of $\CC$ (see \cite{BS1991}, \cite{FS92}). We prove the existence of a leaf, which is injective Brody in $\PP^2$, in the foliation of $J^+$ for certain H\'enon mappings (for the definition of injective Brodyness, see Section \ref{sec:Brody}).
\end{abstract}
\maketitle

\section{Introduction}
The Brody curve first appeared in Brody's proof in \cite{Brody} that every compact non-Kobayashi hyperbolic manifold contains a non-trivial holomorphic image of $\CC$. The non-trivial entire curve is called a Brody curve. (See Section \ref{sec:Brody}.)

Also, in \cite{Gromov}, Gromov considered infinite dimensional geometry and introduced the concept of mean dimension, a topological invariant. As a main example, he considered the space of Brody curves. Recently, the space of Brody curves has been much studied. In particular, for the projective spaces, see \cite{EremenkoArxiv}, \cite{Eremenko2010}, \cite{MatsuoTsukamoto}, \cite{Tsukamoto2008}, \cite{Tsukamoto2009-1}, \cite{Tsukamoto2009} and \cite{Tsukamoto2012}. However, it is not easy to find interesting examples (cf. \cite{Ahn12}, \cite{Ahn14-2}) except trivial ones such as polynomial mappings of $\CC$ to $\PP^2$ and something of that type. From the perspective of geometry, it is interesting to see non-trivial examples.

In this paper, we construct an interesting example of Brody curves using the dynamics of a certain generalized H\'enon mapping.

A generalized H\'enon mapping $f$ is defined simply by a polynomial diffeomorphism of $\CC^2$ of the form $f(z, w)=(p(z)-aw, z)$ where $p$ is a monic polynomial of one complex variable and $a$ is a non-zero constant. The polynomial diffeomorphisms of this class are particularly important since in \cite{FM89}, Friedland and Milnor classified polynomial diffeomorphisms of $\CC^2$ and showed that the only dynamically interesting polynomial diffeomorphisms are the finite compositions of generalized H\'enon mappings. So, many mathematicians have studied them (listing some, cf \cite{BS1991}, \cite{BS1993}, \cite{FS92}, \cite{HO}).

Based on the dynamics of $f$, $\CC^2$ can be divided into two regions. Let
$$
K^+:=\{(z, w)\in\CC^2\colon \exists\, c>0 \textrm{ such that } \|f^n(z, w)\|<c,\,\, \forall\, n\in\NN\}.
$$
Also, let $J^+:=\partial K^+$ and $U^+:=\CC^2\setminus K^+$. Let $g^+:\CC^2\to\RR$ denote the Green function associated to $f$. Then $U^+=\{g^+>0\}$ and $K^+=\{g^+=0\}$.

In \cite{Ahn12} and \cite{Ahn14-2}, the author considered a foliation structure for $U^+$. More precisely, it was proved that the level set $\cL_c:=\{g^+=c\}$ with $c>0$ is foliated by biholomorphic images of $\CC$ and each leaf is dense in $\cL_c$ in \cite{HO}. In \cite{Ahn12} and \cite{Ahn14-2}, the author further proved that every leaf is an injective Brody curve of $\PP^2$ with respect to the Fubini-Study metric. (For definitions, see Section \ref{sec:Brody}.)

In this paper, we consider a similar structural property for $J^+$. In \cite{BS1991}, \cite{BS1993} and \cite{FS92}, the foliation structure of $J^+$ was studied. In particualr, in \cite{BS1991}, Bedford and Smillie proved that $J^+$ admits a foliation $\cF^+$ by biholomorphic images of $\CC$ for a hyperbolic generalized H\'enon mapping $f$. Here, the main theorem of this paper is as follows:
\begin{theorem}\label{thm:mainBrodyJ+}
Let $f(z, w)=(p(z)-aw, z)$ where $p$ is a monic polynomial of one complex variable and $a$ is a non-zero constant. Assume that $f$ is hyperbolic (in the sense of \cite{BS1991}) and $|a|\leq 1$. Then, in the natural foliation $\cF^+$ of $J^+$, there exists a leaf which is an injective Brody curve of $\PP^2$ with respect to the Fubini-Study metric.
\end{theorem}
This is interesting in the sense that if we restrict Theorem \ref{thm:mainBrodyJ+} to the H\'enon mappings in \cite{FS92}, then $J^+$ has fractional Hausdorff dimension and we can have the density property: the closure of the injective Brody curve leaf is equal to $J^+$. In general, these properties may not be true. On the other hand, in the sense that the Fubini-Study metric of a Brody curve is bounded, Theorem \ref{thm:mainBrodyJ+} can be considered as having tame behavior. So, the curve in Theorem \ref{thm:mainBrodyJ+} is a not-too-complicated curve in a complicated set. From these perspectives, this curve is another non-trivial example of injective Brody curves in $\PP^2$.


Note that different from the case of \cite{Ahn12} and \cite{Ahn14-2}, 
due to the recurrent behavior in $J^+$, 
it is not expected that we can locate the final injective Brody curve. We just know the existence of an injective Brody curve leaf in $J^+$.

The main ingredients for Theorem \ref{thm:mainBrodyJ+} are the hyperbolicity of $f$ (in the sense of \cite{BS1991}) and flow-boxes of the foliation of $J^+$ and they are quite different from those for \cite{Ahn12} and \cite{Ahn14-2}. 
\\

\bf
\noindent
Notation.
\rm
We use $\Delta(a, r)$ for the disc in $\CC$ centered at $a\in\CC$ and of radius $r>0$ and $\Delta$ for the standard unit disc in $\CC$.
We denote by $\|\cdot\|$ the standard Euclidean metric of $\CC^2$ and by $ds(P, V)$ the standard Fubini-Study metric on $\PP^2$ of $V\in T_P\PP^2$ at $P\in \PP^2$. For a holomorphic curve $\gamma:U\to\PP^2$ and for $\theta'\in U$, $\|\gamma(\theta)\|_{FS, \theta'}$ denotes $ds(\gamma(\theta'), d\gamma|_{\theta=\theta'}(\frac{d}{d\theta}))$ where $U$ is an open subset of $\CC$.


\section{Preliminaries}\label{sec:prelim}
In this section, we recall some basic properties about generalized H\'enon mappings. A generalized H\'{e}non mapping is a holomorphic polynomial automorphism $f:\CC^2\to\CC^2$ defined by

$$
f(z, w)=(p(z)-aw, z)
$$
where $p(z)$ is a monic polynomial of one complex variable $z$ with degree $d\ge 2$ and $a\neq 0$. Then, $f^{-1}(z, w)=(w, (p(w)-z)/a)$.
\medskip

Let $\PP^2$ be the $2$-dimensional complex projective space and
$$
I_+:=[0:1:0]\quad\textrm{ and }\quad I_-:=[1:0:0]
$$
in the homogeneous coordinate system of $\PP^2$. Then, $f$ has the natural extension to $\widetilde f\colon\PP^2\setminus\{I_+\}\to\PP^2\setminus\{I_+\}$ by
$$
\widetilde f([z:w:t])=\left[t^dp(\frac{z}{t})-awt^{d-1}: zt^{d-1}: t^d\right].
$$
Similarly, $f^{-1}$ also has the natural extension to $\widetilde{f^{-1}}:\PP^2\setminus\{I_-\}\to\PP^2\setminus\{I_-\}$ by
$$
\widetilde{f^{-1}}([z:w:t])=\left[wt^{d-1}: \frac{1}{a}(t^dp(\frac{w}{t})-zt^{d-1}): t^d\right].
$$

We recall the following notions and properties related to the dynamics of $f$ as in \cite{HO}. Let
$$
K^\pm=\{p\in\CC^2\colon \{f^{\pm n}(p)\}\textrm{ is a bounded sequence of }n\}.
$$
Let $J^\pm=\partial K^\pm$, $K=K^+\cap K^-$, $J=J^+\cap J^-$ and $U^\pm=\CC^2\setminus K^\pm$. Then, it is known that $K^\pm$ is closed and $U^\pm$ is open in $\CC^2$.

%

\begin{proposition}[See ~\cite{Sibony99}]\label{prop:I+I-} $K^\pm$, $U^\pm$, $I_\pm$, and $\widetilde{f}$ satisfy the following properties:
\begin{enumerate}
	\item $I_-$ and $I_+$ are the super-attracting fixed points of $\widetilde{f}$ and $\widetilde{f^{-1}}$, respectively,
	\item any compact subset of $U^\pm$ uniformly converges to $I_\mp$, respectively,
	\item $\widetilde{f}(\{t=0\}\setminus I_+)=I_-$ and $\widetilde{f^{-1}}(\{t=0\}\setminus I_-)=I_+$, and
	\item $\overline{K^+}=K^+\cup I_+$ and $\overline{K^-}=K^-\cup I_-$.
\end{enumerate}
\end{proposition}

%

The following theorem describes the behavior of $J^+$.
\begin{theorem}[Theorem 1.3 in \cite{Ahn14-2}]\label{thm:NonexistenceofHoloCurve}
There is no non-trivial holomorphic curve, which passes through $I_+$, and is supported in $\overline{K^+}\subseteq\PP^2$.
\end{theorem}

We recall hyperbolicity for generalized H\'enon mappings in \cite{BS1991} (see \cite{Shub} and also \cite{FS92}). If a generalized H\'enon mapping $f$ is hyperbolic, there are continuous subbundles $E_u$ and $E_s$ such that $T\CC^2_J=E^s\oplus E^u$, and $Df(E^s)=E^s$ and $Df(E^u)=E^u$, and there exists constants $c>0$ and $0<\lambda<1$ such that
\begin{eqnarray*}
\|Df^n|_{E^s}\|<c\lambda^n, && n\geq 0\quad\textrm{ and}\\
\|Df^{-n}|_{E^u}\|<c\lambda^n, && n\geq 0.
\end{eqnarray*}

The Stable Manifold Theorem and Theorem 5.4 in \cite{BS1991} imply that for every $x\in J$, there exists a leaf $\cL_x$ in $\cF^+$ such that $x\in\cL_x$ and $T_x\cL_x=E^s_x$ where $\cF^+$ is the natural foliation of $J^+$ in \cite{BS1991}.
\medskip

\section{Brody Curves}\label{sec:Brody}
In this section, we briefly introduce the concepts of the \emph{Brody curve} and the \emph{injective Brody curve}.

\begin{definition}[Brody Curve]
Let $M$ be a compact complex manifold with a smooth metric $ds_M$. Let $\psi:\CC\to M$ be a non-constant holomorphic map.

The map $\psi$ is said to be \emph{Brody} if $\sup_{\theta\in\CC}ds_M(\psi(\theta), d\psi(\fddtheta{} ))<C_M$ for some constant $C_M>0$. We call the image $\psi(\CC)$ a \emph{Brody curve} in $M$. The curve $\psi(\CC)$ is said to be \emph{injective Brody} if the parametrization $\psi$ is injective.
\end{definition}

In the rest of the paper, we only consider the Brody curves in $\PP^2$ with respect to the standard Fubini-Study metric of $\PP^2$. 
\medskip

Below, we consider some trivial examples. The proofs are all straightforward from computations and so, we omit them.
\begin{proposition}
Let $\alpha$ be a complex constant and $p, q$ polynomials of one complex variable $z$. Then, all curves of the form $[z: p(z): 1]$ and of the form $[p(z)\exp(z): q(z)\exp(\alpha z):1]$ are Brody.
\end{proposition}

However, not all holomorphic curves from $\CC$ to $\PP^2$ are Brody. The mapping $z\to[e^z: e^{iz^2}:1]$ is not Brody. For the verification, simply take $z=bi$ for real $b$ and let $b$ to $\infty$. Even if we require them to be injective, not all injective curves from $\CC$ to $\PP^2$ are Brody. The following gives us some examples of injective but non-Brody curves.

\begin{proposition}
The map $f_n:z\to(z, \exp(z^n))$ is not Brody in $\CC^2\subset\PP^2$ for $n\geq 3$. In particular, not all holomorphic images of $\CC$ in $\PP^2$ are Brody.
\end{proposition}

We close this section by pointing out a property of parametrizations of injective Brody curves.

\begin{proposition}
For an injective Brody curve $\cC$ in $\PP^2$, every parametrization of $\cC$ has uniformly bounded Fubini-Study metrics. In short, the injective Brodyness property does not depend on the choice of the parametrization.
\end{proposition}

\begin{proof}
Let $\phi_1, \phi_2:\CC\to\cC$ be two biholomorphic parametrizations of $\cC$. The composition $\phi_2^{-1}\circ\phi_1:\CC\to\CC$ is a holomorphic automorphism of $\CC$. From a theorem of one complex variable, $\phi_2^{-1}\circ\phi_1(z)=\alpha z+\beta$ for constants $\alpha, \beta\in\CC$ with $\alpha\neq 0$.
\end{proof}

\section{Proof of Theorem \ref{thm:mainBrodyJ+}}
In this section, we prove the main theorem. Basically, we follow the Brody reparametrization lemma. We assume all the hypotheses in Theorem \ref{thm:mainBrodyJ+}.
\medskip

\noindent
\bf
Proof of Proposition \ref{thm:mainBrodyJ+}
\rm
We first define a family of analytic discs. From Corollary 6.13 in \cite{BS1991}, periodic points are dense in $J$. Pick a periodic point $P\in J$ and say $N$ its period. Let $\cL_P$ be a leaf in the foliation $\cF^+$ of $J^+$ passing through $P$ as discussed in Section \ref{sec:prelim}. Fix an analytic disc $\psi:\Delta\to\cL_P$ such that $\psi(0)=P$ and $\|\psi\|_{FS, 0}>0$. Then we consider a family of analytic discs as follows:
$$
\varphi_n:=f^{-Nn}\circ \psi:\Delta\to\cL_P.
$$
Then, since $\cL_P$ is a stable manifold of $P$, from the hyperbolicity of $f$, $\|\varphi_n\|_{FS, 0}\to\infty$ as $n\to\infty$.
\smallskip

Now we apply the Brody reparametrization lemma. Note that $\varphi_n$'s are holomorphic in a slightly larger disc. Define $H_n:\Delta\to\RR^+$ by $H_n(\theta):=\|\varphi_n\|_{FS, \theta}(1-|\theta|^2)$. Then, there exists $\theta_n\in\Delta$ such that $H_n(\theta_n)=\max_{\theta\in\Delta}H_n(\theta)$. For each $n$, define a M\"obius transformation $\mu_n(\zeta):=(\zeta+\theta_n)/(1+\overline{\theta_n}\zeta)$ mapping $0$ to $\theta_n$. Let $g_n:=\varphi_n\circ\mu_n$. Then
$$
\|g_n\|_{FS, \zeta}(1-|\zeta|^2)=\|\varphi_n\|_{FS, \theta}|\mu_n'(\zeta)|(1-|\zeta|^2)=\|\varphi_n\|_{FS, \theta}(1-|\theta|^2).
$$
So, $ \|g_n\|_{FS, \zeta}\leq \|g_n\|_{FS, 0}/(1-|\zeta|^2)$. Let $R_n=\|g_n\|_{FS, 0}$ and define $k_n(\theta)=g_n(\theta/R_n)$. Then,
$$
\|k_n\|_{FS, \theta}=\frac{\|g_n\|_{FS, \theta/R_n}}{R_n}\leq\frac{\|g_n\|_{FS, 0}}{R_n(1-|\theta/R_n|^2)}\leq 2,
$$
on $\Delta(0, R_n/2)$. Note that $\|k_n\|_{FS, 0}=1$ and that from the hyperbolicity of $f$, we see that $R_n\to\infty$ as $n\to\infty$. Hence, from a normal family argument and the compactness of $\PP^2$, we can find a non-trivial holomorphic map $\Phi:\CC\to \overline{J^+}$
such that $\|\Phi\|_{FS, 0}=1$. From Proposition \ref{prop:I+I-}, we have $\overline{K^+}=K^+\cup I_+$. However, Theorem \ref{thm:NonexistenceofHoloCurve} implies that $\Phi(\CC)\subset J^+$.
\smallskip

We prove that the Brody curve $\Phi(\CC)$ sits in a single leaf of the foliation of $J^+$. Suppose the contrary. Then, there exists two points $\alpha, \beta\in \CC$ such that $\Phi(\alpha)$ and $\Phi(\beta)$ live in two different leaves and the two $\alpha$, $\beta$ are sufficiently close so that some fraction of the complex curve $\Phi(\CC)$ connecting $\Phi(\alpha)$ and $\Phi(\beta)$ sits in a single flow-box of the foliation of $J^+$. Let $\gamma\subset\Phi(\CC)$ denote the complex curve connecting $\Phi(\alpha)$ and $\Phi(\beta)$. Then, there exists a constant $\epsilon>0$ such that for any plaque $T$ in the flow-box, $\sup_{(z, w)\in\gamma} \dist((z, w), T)>\epsilon$ where $\dist(\cdot, \cdot)$ is with respect to the standard Euclidean distance of $\CC^2$. This is a contradiction to the local uniform convergence since the image of each reparametrized analytic disc sits inside a leaf of the foliation $\cF^+$ of $J^+$.
\smallskip

We show that $\Phi$ is one-to-one. Suppose to the contrary that $\Phi$ is not one-to-one. Then, there are $\alpha, \beta\in\CC$ and $q\in\Phi(\CC)$ such that $\alpha\neq \beta$ and $\Phi(\alpha)=\Phi(\beta)=q$. Consider a sufficiently large $R_q> 1$ such that $\alpha, \beta\in\Delta(0, R_q)$. Let $F$ be a compact set of $\CC^2$ such that its interior contains $J$. Consider a finite covering of $J^+\cap F$ consisting of flow-boxes of $\cF^+$. Note that since $|a|\leq 1$, Theorem 5.9 in \cite{BS1991} says that for any leaf $\cL$ in $J^+$, there exists a point $x\in J$ such that $\cL$ is a stable manifold of the point $x$. Since $\Phi(\Delta(0, 2R_q))$ live in a single leaf, there exists sufficiently large $N_q\in\NN$ such that the analytic disc $f^{N_q}(\Phi(\Delta(0, 2R_q)))$ sits inside a flow-box in the finite covering. Again, since the image of each reparametrized analytic disc sits inside a leaf of the foliation of $J^+$, the injectivity of limit maps is implied by the local uniform convergence to $\Phi$ of a subsequence $\{k_{n_j}\}$ of injective maps, projection onto the base direction in the flow-box, and the Hurwicz theorem.
\smallskip

Since there is no proper biholomorphic image of $\CC$ inside $\CC$, the leaf containing the injective Brody curve itself is an injective Brody curve. This proves our theorem.
\hfill $\Box$


\end{document}